\documentclass[runningheads]{llncs}
\usepackage{graphicx,color}
\usepackage{amsmath}
\usepackage{amssymb}
\usepackage{epstopdf}
\usepackage{hhline}
\usepackage{multirow}

\newif\ifshowcomments
\ifshowcomments
    \newcommand\sam[1]{\textcolor{green}{\textbf{ S: #1}}}
    \newcommand\luca[1]{\textcolor{red}{\textbf{ L: #1}}}
    \newcommand\lena[1]{\textcolor{purple}{\textbf{ Le: #1}}}
    \newcommand\stefan[1]{\textcolor{blue}{\textbf{ St: #1}}}
\else
    \newcommand\sam[1]{}
    \newcommand\luca[1]{}
    \newcommand\lena[1]{}
    \newcommand\stefan[1]{}
\fi

\usepackage{enumerate}
\usepackage{subcaption}
\usepackage{caption}
\usepackage{tikz}

\newcommand\restr[2]{{%
  \left.\kern-\nulldelimiterspace %
  #1 %
  \vphantom{\big|} %
  \right|_{#2} %
  }}

\DeclareMathOperator{\diag}{diag}
\DeclareMathOperator{\ocp}{ocp}

\title{Total Matching and Subdeterminants}
\author{Luca Ferrarini\inst{1}\orcidID{0000-0002-8903-2871} \and Samuel Fiorini\inst{2}\orcidID{0000-0002-6845-9008} \and Stefan Kober\inst{2}\orcidID{0000-0003-2610-1494} \and Yelena Yuditsky\inst{2}\orcidID{0000-0002-6467-3437}}
\institute{\'Ecole des Ponts ParisTech \email{luca.ferrarini@enpc.fr} \and Universit\'e libre de Bruxelles \email{\{Samuel.Fiorini,Stefan.Kober,Yelena.Yuditsky\}@ulb.be}}
\date{\today}

\begin{document}

\maketitle

\begin{abstract}
In the total matching problem, one is given a graph $G$ with weights on the vertices and edges. The goal is to find a maximum weight set of vertices and edges that is the non-incident union of a stable set and a matching.
We consider the natural formulation of the problem as an integer program (IP), with variables corresponding to vertices and edges. Let $M = M(G)$ denote the constraint matrix of this IP. We define $\Delta(G)$ as the maximum absolute value of the determinant of a square submatrix of $M$.
We show that the total matching problem can be solved in strongly polynomial time provided $\Delta(G) \leq \Delta$ for some constant $\Delta \in \mathbb{Z}_{\ge 1}$. We also show that the problem of computing $\Delta(G)$ admits an FPT algorithm. We also establish further results on $\Delta(G)$ when $G$ is a forest.
\end{abstract}

\section{Introduction}
Let $G = (V(G),E(G))$ be a simple graph\footnote{We can assume that $G$ is simple without loss of generality.}. Given weights $w(v) \in \mathbb{R}$ for each vertex $v $ and $w(e) \in \mathbb{R}$ for each edge $e$, the natural integer programming formulation of the total matching problem is as follows\footnote{We use $\delta(v)$ to denote the set of edges incident to vertex $v$.}:
$$
\begin{array}{rl}
\max & \sum_{v \in V(G)} w(v) x_v + \sum_{e \in E(G)} w(e) y_e\\
\text{s.t.} & x_v + \sum_{e \in \delta(v)} y_e \leq 1 \quad \forall v \in V(G)\\
& x_v + x_w + y_e \leq 1 \quad \forall e = vw \in E(G)\\
& x \geq \mathbf{0},\ y \geq \mathbf{0}\\
& x, y \text{ integer.}
\end{array}
$$

We denote by $M(G)$ the constraint matrix of the above integer program, ignoring the nonnegativity constraints. That is, we let
$
M(G) := \begin{bmatrix}
I & B\\
B^\intercal & I\\
\end{bmatrix}
$,
where $B = B(G) \in \{0,1\}^{V(G) \times E(G)}$ denotes the (vertex-edge) incidence matrix of $G$ and $I$ denotes the identity matrix of appropriate size.

We say that two elements (vertices or edges) $\alpha, \beta \in V(G) \cup E(G)$ are \emph{incident} if $\alpha = \beta$ or if one is an edge and the other is an endpoint of that edge. %Thus, every vertex $v$ is incident to itself and all the edges in $\delta(v)$, and every edge $e = vw$ is incident to itself, $v$ and $w$.
Notice that $M(G)$ is simply the matrix of that incidence relation on $V(G) \cup E(G)$.

We let $\Delta(G)$ denote the maximum subdeterminant of $M(G)$, that is, the maximum of $|\det M'|$ taken over all square submatrices $M'$ of $M(G)$. %(Notice that the maximum subdeterminant of $M(G)$ would not change if we included the $|V(G)| + |E(G)|$ extra rows for the nonnegativity inequalities.)

The study of the total matching problem has been initiated by Alavi, Bezhad, Lesniak-Foster, and Nordhaus~\cite{alavi_1977}. 
As a generalization of the stable set problem, it is clearly NP-complete. 
In fact, the problem is NP-hard already for bipartite and planar graphs~\cite{manlove_1999}.
%, however, the author in the same paper shows that the problem can be computed in polynomial time for trees. 
However efficient algorithms are only known for very structured classes of graphs, such as trees, cycles, or complete graphs~\cite{manlove_1999,leidner_2012}.
Recently the first polyhedral study of the total matching problem was introduced~\cite{Polyhedra}.

Before exploring the total matching problem further, we discuss our motivations to study this problem through the prism of subdeterminants. It is conjectured that the IP $\max \{w^\intercal x : Ax \le b,\ x \in \mathbb{Z}^n\}$ can be solved in strongly polynomial time whenever the maximum subdeterminant of the constraint matrix $A$ is bounded by a constant $\Delta \in \mathbb{Z}_{\ge 1}$. 

If $\Delta = 1$ then $A$ is \emph{totally unimodular} and solving the IP is equivalent to solving its LP relaxation $\max \{w^\intercal x : Ax \le b\}$, and can hence be done in strongly polynomial time. This is a foundational result in discrete optimization, with countless applications. Artmann, Weismantel and Zenklusen \cite{artmann_2017} proved that this generalizes to $\Delta = 2$. To this day, the general conjecture remains open for all $\Delta \ge 3$, despite recent efforts \cite{artmann_2016,gribanov_2016b,naegele_2019,gribanov_2022,glanzer_2022,naegele_2022,naegele_2023}. However, it is known to hold in some special cases, for instance when $A$ has at most two nonzeroes per row or per column \cite{fiorini_2022}.

While the conjecture remains out of reach, it is worthwhile to \emph{test} it on interesting problems, and the total matching is one of them. One of the expected outcomes of such research efforts is to understand better how the structure of the constraint matrix $A$ relates to the maximum subdeterminant, and more specifically what a constant bound on the maximum subdeterminant entails in terms of forbidden submatrices of $A$. 

For instance, incidence matrices of odd cycles have determinant $\pm 2$, and unsurprisingly they are relevant to our work. However, it turns out that incidence matrices of \emph{near-pencils} play an even more important role. These are defined as follows. For $k \in \mathbb{Z}_{\ge 1}$, let $N_k$ denote the $(1+k) \times (1+k)$ matrix defined as
$
N_k :=
\begin{bmatrix}
    1 &\mathbf{1}^\intercal\\
    \mathbf{1} &I
\end{bmatrix}
$.
This matrix is the incidence matrix of a \emph{near-pencil} of order $k$. Notice that $\det N_k = 1-k$. 

We conclude this introduction with an overview of the paper. We first consider the case of general graphs in Section~\ref{sec:general} and prove our main algorithmic results, which are a $O(2^{O(\Delta \log \Delta)} (|V(G)|+|E(G)|))$ time algorithm to solve the total matching problem on graphs $G$ with $\Delta(G) \le \Delta$ and a $O(|V(G)|+|E(G)| + 2^{O(\Delta \log \Delta)})$ time algorithm to either correctly report that $\Delta(G) > \Delta$ or compute $\Delta(G)$ exactly. In Section~\ref{sec:forests}, we study further properties of $\Delta(G)$ when $G$ is a forest. In Section~\ref{sec:conclusion}, give an outlook on the questions solved in this paper and also on future research. %One of the main questions that are left open in this work is the complexity of computing $\Delta(G)$ when no \emph{a priori} bound is specified, for general graphs and even for forests.

\section{General graphs} \label{sec:general}

In this section, we first prove some structural results that relate $\Delta(G)$ to the structure of $G$ (in Section~\ref{sec:structure}), and then use these results to obtain our main two algorithmic results (in Section~\ref{sec:algorithms}).

\subsection{Maximum subdeterminants and graph structure} \label{sec:structure}

\begin{lemma}\label{lemma:components}
If $G_1$, \ldots, $G_c$ denote the components of graph $G$, then
$$
\Delta(G) = \prod_{i=1}^c \Delta(G_i)\,.
$$
\end{lemma}

\begin{proof}
    Consider any ordering of the elements (vertices and edges) of $G$ such that the elements of each component $G_i$ are consecutive. With respect to such an ordering, $M(G)$ has a block-diagonal structure, with one block per component. The result follows.\qed 
\end{proof}

The next two lemmas implicitly use the fact that $\Delta(G)$ is monotone under taking subgraphs: if $G \subseteq H$ then $\Delta(G) \le \Delta(H)$.

\begin{lemma}\label{lemma:degree}
    Let $G$ be a graph and let $d_\mathrm{max}(G)$ denote the maximum degree of $G$. Then 
    $$
    \Delta(G) \ge d_\mathrm{max}(G) - 1\,.
    $$
\end{lemma}

\begin{proof}
    If $G$ contains a vertex of degree $k \ge 1$, then $M(G)$ contains $N_k$ as a submatrix. Hence, $\Delta(G) \geq |\det N_k| = k-1$. \qed 
\end{proof}

Below, we use the notation $C = v_1 e_1 v_2 e_2 \cdots v_n e_n v_1$ to denote the cycle with vertex set $V(C) = \{v_1,v_2,\ldots,v_n\}$ and edge set $E(C) = \{e_1,e_2,\ldots,e_n\}$ where $e_i = v_iv_{i+1}$ for $i \in [n-1] = \{1,2,\ldots,n-1\}$ and $e_n = v_nv_1$. Later, we use a similar notation for paths.

\begin{lemma} \label{lemma:cycle_LB_2}
    Let $k\in \mathbb{N}$. If a graph $G$ contains $k$ disjoint cycles, then $\Delta(G)\ge 2^k$. 
\end{lemma}

\begin{proof}
    Firstly we consider the case where $k=1$. 
    We may assume without loss of generality that $G$ is a cycle. If $G$ is an odd cycle, then its incidence matrix $B = B(G)$ has determinant $\pm 2$ and hence $\Delta(G) \geq |\det B| = 2$. 
    
    Now assume that $G := v_1 e_1 v_2 e_2 \cdots v_n e_n v_1$ is an even cycle. Consider the submatrix $M'$ of $M(G)$ whose rows are indexed by $V(G) \cup \{e_n\}$ and whose columns are indexed by $E(G) \setminus \{e_n\} \cup \{v_1,v_n\}$. It is easy to see that $M'$ is the incidence matrix of an odd cycle, hence we get again $\Delta(G) \geq |\det M'| = 2$.

    The bound for $k\ge 2$ follows from the above and Lemma~\ref{lemma:components}.
\qed \end{proof}

\begin{lemma}\label{lemma:TUMpath}
$M(G)$ is totally unimodular if and only if every component of $G$ is a path.
\end{lemma}
    
\begin{proof}
    First, assume that every component of $G$ is a path. By Lemma~\ref{lemma:components}, we may assume that $G$ is a path, say $G = v_1 e_1 v_2 e_2 \cdots e_{n-1} v_n$. Using this ordering of its elements, we see that $M(G)$ has the consecutive ones property, and is hence totally unimodular.

    Conversely, assume that $M(G)$ is totally unimodular. By Lemma~\ref{lemma:degree}, the maximum degree of $G$ is at most $2$. By Lemma~\ref{lemma:cycle_LB_2}, every component of $G$ is a path.
\qed \end{proof}

Let $G$ be any graph. Consider the corresponding constraint matrix $M = M(G)$, and a \emph{witness}, i.e., a square submatrix $M'$ of $M$ whose determinant is maximum in absolute value. It is natural to assume that every proper square submatrix $M''$ of witness $M'$ has $|\det M''| < |\det M'|$, in which case we say that $M'$ is \emph{minimal}. In particular, if $M'$ is a minimal witness then every row and column of $M'$ has at least two ones. 

In order to facilitate the analysis, we use the following terminology for representing the witness $M'$. We say that an element of $G$ (vertex or edge) is \emph{red} if the corresponding row is selected as a row of $M'$, and \emph{cyan} if the corresponding column is selected as a column of $M'$. Notice that an element can be red and cyan at the same time, in which case we say that it is \emph{bichromatic} (or \emph{bicolored}). The elements that belong to a single color are said to be \emph{monochromatic}. We allow some elements to be \emph{uncolored}. 
Clearly the square submatrices of $M$ correspond to the colorings of the elements of $G$ with $k$ red elements and $k$ cyan elements, where $k$ is a positive integer.

We say that witness $M' \in \{0,1\}^{k \times k}$ has a \emph{fault} if $M'$ has two rows (resp.\ columns) such that, regarded as sets, one is contained in the other and the two differ by at most one element. A minimal witness $M'$ cannot have a fault, since either $\det M' = 0$ or one can remove one row and column of $M'$ to obtain a submatrix $M'' \in \{0,1\}^{(k-1) \times (k-1)}$ such that $|\det M''| = |\det M'|$. This can easily be seen by performing a row (resp.\ column) operation on $M'$.

\begin{lemma} \label{lem:technical}
    Let $G$ be a graph. Let $x_0, x_1, x_2, x_3, x_4$ be elements of $G$ such that each $x_i$, $i \in [3]$ is incident to $x_{i-1}, x_i, x_{i+1}$ and to no other element of $G$ (possibly $x_0 = x_4$, but the elements are otherwise distinct). Consider a minimal witness $M'$ for $G$ and the corresponding coloring. The following hold:
    \begin{enumerate}[(i)]
        \item $x_1$ is uncolored if and only if $x_2$ is uncolored;
        \item  if $x_2$ is bicolored and $x_1, x_2, x_3$ are all colored in the same color, then each of $x_0, x_1, x_2, x_3, x_4$ is colored in the opposite color.
    \end{enumerate}
\end{lemma}

\begin{proof}
    (i) Toward a contradiction, suppose that $x_1$ is uncolored, but $x_2$ is colored, say red. Hence, $x_2$ and $x_3$ are both cyan. Looking at the corresponding columns of $M'$, we see that $M'$ has a fault, a contradiction. This shows that $x_2$ is uncolored whenever $x_1$ is. The same argument shows that $x_1$ is uncolored whenever $x_2$ is.

    (ii) Assume for instance that $x_1$, $x_2$ and $x_3$ are all red, and $x_2$ is also cyan. In order to avoid a fault in the rows for $x_1$ and $x_2$, elements $x_0$ and $x_3$ need to be cyan. Similarly, by the considering the rows for $x_2$ and $x_3$ we conclude that $x_1$ and $x_4$ are cyan. 
\qed \end{proof}

Now consider any sequence $x_0,x_1,x_2,\ldots,x_k$ ($k \ge 4$) of elements of $G$ such that for each $i \in [k-1]$, elements $x_{i-1}, x_i, x_{i+1}$ are distinct and the only elements of $G$ incident to $x_i$. Notice that such sequences essentially correspond to induced paths in $G$. 

More precisely, assume that $x_0$ and $x_k$ are vertices of $G$ (hence $k$ is even). Then, $P := x_2 x_3 \cdots x_{k-2}$ is an induced path in $G$. 

Applying  Lemma~\ref{lem:technical} several times, we see that if one element of $P$ is uncolored, then they all are. Moreover, if some element $x_i$ of $P$ is bicolored with $x_{i-1}, x_{i}, x_{i+1}$ of the same color then all elements of $P$ are bicolored. These observations form the basis of the next two lemmas.

\begin{lemma} \label{lem:long_induced_path}
Let $G$ be a graph containing a path $P$ of length $7$ all whose vertices have degree $2$ in $G$. If $G'$ denotes the graph obtained from $G$ by contracting $P$ to a single vertex, then $\Delta(G) = \Delta(G')$.
\end{lemma}

\begin{proof}
Let $P = v_1 e_1 v_2 e_2 \cdots e_6 v_7$, where $v_i$, $i \in [7]$ denote the vertices of $P$ and $e_i$, $i \in [6]$ denote the edges of $P$. Let $e_0$ (resp.\ $e_7)$ denote the other edge of $G$ incident to $v_1$ (resp.\ $v_7$). Similarly, let $v_0$ (resp.\ $v_8$) denote the end of $e_0$ distinct from $v_1$ (resp.\ of $e_7$ distinct from $v_7$). Possibly, $v_0 = v_8$.

Again, consider a minimal witness $M'$ for $G$ and the corresponding coloring. By Lemma~\ref{lem:technical}, there are three cases to consider.\medskip

\noindent{}\emph{Case 1}. All the elements of $P$ are uncolored. Then $M'$ is a submatrix of $M(G')$, hence
$\Delta(G) = \Delta(G')$. 
\medskip

\noindent{}\emph{Case 2}. All the elements of $P$ are bicolored. Then there exists a square submatrix $M''$ of $M(G')$ with $|\det M''| = |\det M'|$. (Here, we use the fact that $P$ has $6$ edges, and that $6$ is a multiple of $3$.) We have once more that $\Delta(G) = \Delta(G')$.
\medskip

\noindent{}\emph{Case 3}. Every row and column of $M'$ corresponding to an element of $P$ has exactly $2$ ones. Again, there exists a square submatrix $M''$ of $M(G')$ with $|\det M''| = |\det M'|$, which implies $\Delta(G) = \Delta(G')$. (This time we use the fact that $6$ is a multiple of $2$.)
\qed \end{proof}

\begin{lemma}\label{lemma:cycle}
Let $G$ be a $n$-cycle. Then, 
\begin{equation*}
   \Delta(G)=\left\{
  \begin{array}{@{}ll@{}}
    3 & \text{if}\ n \equiv 1,2 \pmod 3, \\
    2 & \text{if}\ n \equiv 0 \pmod 3.
  \end{array}\right.
\end{equation*}
\end{lemma}

\begin{proof}
    Consider a minimal witness $M' \in \{0,1\}^{k \times k}$ for $G$ and the corresponding coloring. By Lemma~\ref{lem:technical}, either $M' = M(G)$ or $M'$ is a submatrix of $M(G)$ having exactly two ones in each row and column (see also the proof of Lemma~\ref{lemma:cycle_LB_2}).
    In the first case, $k = 2n$ and $n \equiv 1,2 \pmod 3$ and $| \det M' | = 3$ (if $n \equiv 0 \pmod 3$ then $| \det M' | = 0$, contradicting the fact that $M'$ is a witness). In the second case, $k \equiv 1 \pmod 2$ and $| \det M' | = 2$ (if $k \equiv 0 \pmod 2$ then again $| \det M' | = 0$, a contradiction).
\qed \end{proof}

\begin{lemma}\label{lemma:LB}
Let $G$ be a graph and let $D \subseteq V(G)$ be a set of vertices which have at least $2$ neighbours in $\overline{D} := V(G) \setminus D$. Let $H$ denote the bipartite graph obtained from $G$ by deleting all edges contained in $D$ or $\overline{D}$. Then 
\[
\Delta(G)\ge \prod_{v\in D} \left(d_{H}(v)-1\right)\,.
\]
\end{lemma}

\begin{proof}
Let $M'$ denote the principal submatrix of $M(G)$ whose rows and columns are indexed by vertices $v \in D$ and edges in $\delta_H(v)$ for $v \in D$. After reordering the rows (resp.\ columns) of $M'$ so that the row (resp.\ column) of each vertex $v \in D$ is immediately followed by the rows (resp.\ columns) for the edges in $\delta_H(v)$ (in any order), we see that $M'$ has a block-diagonal structure, with one block per vertex of $D$. Moreover, the block for $v \in D$ in $M'$ is the near-pencil matrix $N_k$ with $k := d_H(v)$. It follows that 
$
\Delta(G) \geq |\det M'| = \prod_{v\in D} \left(d_{H}(v)-1\right)\,.
$
\qed \end{proof}

%\begin{obs}\label{lemma:star}
%Let $S$ be a star with $d$ leaves, then $M(S)$ contains $J_d$ as a submatrix, and in particular $\Delta(M)\ge d-1$. 
%\end{obs}

\begin{lemma}\label{lemma:degBound}
Every graph $G$ has at most $7 \log \Delta(G)$ vertices of degree at least $3$.
\end{lemma}

\begin{proof}
Consider any inclusion-wise maximal packing of vertex-disjoint $K_{1,3}$ subgraphs in $G$, and let 
$X$ denote the set of vertices covered by the packing. By Lemma~\ref{lemma:LB}, we have that $|X| \le 4 \log \Delta(G)$. 

By maximality of our packing, all degrees in $G - X$ are at most $2$. Hence, $G-X$ has a proper $3$-coloring. Let $Y \subseteq V(G-X)$ denote the set of vertices of $G-X$ whose degree in $G$ is at least $3$. Considering the color class that contains the largest number of vertices of $Y$, and using again Lemma~\ref{lemma:LB}, we see that $|Y| \leq 3 \log \Delta(G)$. 

We conclude that the total number of vertices of degree at least $3$ in $G$ is at most $|X| + |Y| \leq 4 \log \Delta(G) + 3 \log \Delta(G) = 7 \log  \Delta(G)$.
\qed \end{proof}

\begin{lemma}\label{lem:general_structure}
Let $\Delta \in \mathbb{Z}_{\geq 1}$ be a constant, and let $G$ be any graph such that $\Delta(G) \le \Delta$. There exists a vertex subset $Z$ such that $G - Z$ is a disjoint union of paths, $|Z| \leq 8 \log \Delta$ and $|\delta(Z)| \le 8 (\Delta + 1) \log \Delta$. Moreover, $Z$ can be computed in time $O(|V(G)|+|E(G)|)$.
\end{lemma}

\begin{proof}
Let $X := \{v \in V(G) \mid d(v) \geq 3\}$ denote the set of vertices of degree at least $3$. By Lemma~\ref{lemma:degBound}, we have $|X|\le 7 \log \Delta$. Obviously, $G-X$ is a disjoint union of paths and cycles. By Lemma~\ref{lemma:cycle_LB_2}, there are at most $\log \Delta$ cycles in $G - X$ (all disjoint). Hence there exists a set $Y \subseteq V(G-X)$ of size $|Y| \le \log \Delta$ intersecting every cycle of $G - X$. Letting $Z := X \cup Y$, we see that $G-Z$ is a disjoint union of paths.
By Lemma~\ref{lemma:degree}, the degree of any vertex in $G$ is at most $\Delta+1$, hence the number of edges in the cut $\delta(Z)$ is at most $(\Delta+1) |Z| \leq 8(\Delta+1)\log \Delta$. 

It is easy to see that the set $Z$ can be found in time linear in the size of the graph. 
\qed \end{proof}

\subsection{Algorithmic results} \label{sec:algorithms}

\begin{theorem}
Let $\Delta \in \mathbb{Z}_{\geq 1}$ be a constant. There is an algorithm to find a total matching of maximum weight in a graph $G$ such that $\Delta(G) \le \Delta$ in time $O(2^{O(\Delta \log \Delta)}(|V(G)|+|E(G)|))$.
\end{theorem}

\begin{proof}
Let $Z \subseteq V(G)$ be as in Lemma~\ref{lem:general_structure}. Let $I = I(Z)$ denote the set of elements of $G$ that are incident to some vertex of $Z$, including the edges with both ends in $Z$. Notice that $|I| = O(\log \Delta) + O(\Delta \log \Delta) + O(\log^2 \Delta)=O(\Delta \log \Delta)$.

We find a total matching of maximum weight by considering, one after the other, each subset total matching $M_1 \subseteq I$. For a fixed $M_1$, let $Z'$ denote the set of vertices of $G - Z$ that are incident to some edge in $M_1$. Notice that by removing from $G-Z$ all the vertices of $Z'$ (as well as the incident edges), we delete all the elements of $G - Z$ that are incident to some element of $M_1$ (and hence cannot be included in any total matching of $G$ extending $M_1$). Let $M_2$ be a maximum weight total matching in $G - Z - Z'$. Such a total matching can be computed in linear time via dynamic programming, since $G - Z - Z'$ is a disjoint union of paths. 

The optimal solution returned by the algorithm is the best total matching of the form $M_1 \cup M_2$ over all choices of a total matching $M_1 \subseteq I$. The number of choices for $M_1$ is bounded by $2^{O(\Delta \log \Delta)}$. Hence the running time of the above algorithm is $O(2^{O(\Delta \log \Delta)} (|V(G)|+|E(G)|))$.
\qed \end{proof}

\begin{theorem}\label{thm:recognition}
Let $\Delta \in \mathbb{Z}_{\geq 1}$ be a constant. There is an algorithm that, given any graph $G$, either computes $\Delta(G)$ or correctly reports that $\Delta(G) > \Delta$ in time $O(|V(G)|+|E(G)|+2^{O(\Delta \log \Delta)})$.
\end{theorem}

\begin{proof}
By Lemma~\ref{lem:general_structure} in time $O(|V(G)|+|E(G)|)$ we either find a set $Z$ as in the lemma or report correctly that $\Delta(G) > \Delta$. 

Assume that we are in the former case. By Lemmas~\ref{lemma:components} and \ref{lemma:TUMpath}, every component of $G-Z$ that is not connected to $Z$ does not contribute to $\Delta(G)$, and can be deleted from $G$ without changing $\Delta(G)$. There are at most $|\delta(Z)| = O(\Delta \log \Delta)$ components of $G - Z$ that send an edge to $Z$. From the proof of Lemma~\ref{lem:general_structure}, we see that each such component is a path whose vertices have all degree at most $2$ in $G$.

By Lemma~\ref{lem:long_induced_path} we can shrink each component of $G-Z$ to a path of bounded length without changing $\Delta(G)$. This reduces $G$ to a graph with $O(\Delta \log \Delta)$ vertices and edges. The maximum subdeterminant of $M(G)$ can now be computed in time $O(2^{O(\Delta \log \Delta)})$, by brute force. The result follows.
\qed \end{proof}

% \sam{I have a stupid remark: if $\Delta = 1$ then $\log \Delta = 0$ and $O(\Delta \log \Delta) = 0$! Should we worry about this?} 
% \lena{yes, this is a bit annoying, we could restrict $\Delta$ to be at least 2?} \sam{Actually, this seems to never cause a problem. False alarm!}

% \sam{Is computing $\Delta(G)$ NP-hard for general graphs? What about $d$-regular graphs and Lemma~\ref{lemma:LB}? Relate this to maximum independent set? MAXIS gives a lower bound of $(d-1)^{\alpha(G)}$, what about upper bound?}

\section{Forests} \label{sec:forests}

In this section, we explore several properties of $\Delta(G)$ for forests. Let $G$ be a forest, and let $M'$ be a minimal witness for $G$. As before, consider the corresponding coloring of the elements of $G$. Our first lemma states that $M'$ is a principal submatrix of $M$.

\begin{lemma} \label{lem:no_monochromatic_element}
No element of $G$ is monochromatic.
\end{lemma}

\begin{proof} 
%\sam{Can we shorten this perhaps? Stefan has ideas to do so.} 
Toward a contradiction, assume that some element $\alpha$ is monochromatic, say red. 

We claim that, for some $k \geq 2$, we can partition the red elements into $k+1$ sets, $\{\alpha\}$, $R_1$, \ldots, $R_k$, and the cyan elements into $k$ sets $C_1$, \ldots, $C_k$ such that (i) no element of $R_i$ is incident to an element of $C_j$ for $i \neq j$, and (ii) $\alpha$ is incident to at most one element in each $C_j$. In other words, we claim that, after permuting the rows and columns, $M'$ takes the form
$$
M' = 
\begin{bmatrix}
a_1^\intercal & a_2^\intercal & \cdots & a_k^\intercal\\
M'_1          & 0             & \cdots & 0\\
0             &M'_2           & \cdots & 0\\
\vdots        &\vdots         & \ddots & \vdots\\
0             &0              & \cdots & M'_k
\end{bmatrix}
$$
where each $a_j$ is a 0/1 vector of the appropriate dimension, with at most one $1$.

To prove the claim, we consider separately the two cases that can occur. 

First, assume that $\alpha = e = v_1v_2$ is an edge. By minimality of $M'$, both $v_1$ and $v_2$ are cyan. In this case, we take $k = 2$ and consider an arbitrary decomposition of forest $G$ as the union of $(\{v_1,v_2\},\{e\})$ and two vertex-disjoint forests $G_1$ and $G_2$, the first containing $v_1$ and the second $v_2$. Then, for $i = 1, 2$, we let $R_i$ (resp.\ $C_i$) be the set of red (resp.\ cyan) elements in $G_i$. It remains to check that (i) and (ii) are satisfied. Clearly, no element of $R_i$ is incident to an element of $C_j$ for $i \neq j$. Moreover, $e$ is incident to exactly one element of $C_1$ (namely, $v_1$) and exactly one element of $C_2$ (namely, $v_2$).

Second, assume that $\alpha = v$ is a vertex. Let $e_1$, \ldots, $e_k$ denote the edges incident to $v$ (by minimality of $M'$, at least two of these edges are cyan, which implies in particular that $k \geq 2$). In this case, we decompose forest $G$ as the union of $k$ forests $G_1$, \ldots, $G_k$ all containing $v$ but otherwise vertex-disjoint, in such a way that each $G_i$ contains edge $e_i$, for $i = 1, \ldots, k$. Then, for $i = 1, \ldots, k$, we let $R_i$ denote the red elements distinct from $v$ in $G_i$ and $C_i$ denote the cyan elements in $G_i$. Again, properties (i) and (ii) are satisfied. This concludes the proof of the claim.

Now, for $i = 1, \ldots, k$, let $r_i := |R_i|$ and $c_i := |C_i|$. We have $r_i \leq c_i$ for each $i$ since the $r_i$ rows of $M'$ for the elements of $R_i$ are linearly independent vectors spanning a space of dimension at most $c_i$. Also, we have $c_i \leq r_i+1$ for each $i$ since the $c_i$ columns of $M'$ for the elements of $C_i$ are linearly independent vectors spanning a space of dimension at most $r_i + 1$. Hence, $c_i - r_i \in \{0,1\}$ for $i = 1, \ldots, k$. Since $M'$ is square, we have
$$
1 + \sum_{i=1}^k r_i = \sum_{i=1}^k c_i \implies 
\sum_{i=1}^k \underbrace{(c_i - r_i)}_{\in \{0,1\}} = 1\,.
$$
After permuting the indices if necessary, we may assume that $c_1 = 1 + r_1$ and $c_i = r_i$ for $i \geq 2$. 

We find
$$
\det M' = \det  
\begin{bmatrix}
a_1^\intercal\\
M'_1
\end{bmatrix}
\det M'_2 \cdots \det M'_k
$$
By choice of $M'$, none of these determinants is zero. In particular, $\alpha$ is incident to exactly one element of $C_1$, say $\beta$. Notice that the first determinant equals (in absolute value) the determinant of the square submatrix associated to $R_1$ and $C_1 \setminus \{\beta\}$. This means that we can uncolor element $\alpha$ and remove the color cyan from $\beta$ and find the same determinant, which contradicts the minimality of $M'$.
\qed \end{proof}

%\sam{TO DO: restructure the story. First say what we need to do in general, and then Schur complements, and finally discuss uncolored edges and uncolored vertices with examples.}
%In light of Lemma~\ref{lem:no_monochromatic_element}, we can view the search for a minimal witness $M'$ as a two stage process. In the first stage, one selects a set of edges to bicolor. Notice that leaving some edge $e$ uncolored has the same effect as deleting $e$ from $G$. In some cases, it is actually beneficial to do this, see Example~\ref{ex:first}. In the second stage, one selects a set of vertices to bicolor. By minimality of $M'$, every such vertex is incident to at least one bichromatic edge, and every bichromatic edge is incident to at least one bichromatic vertex.

Resuming our discussion of the minimal witness $M'$, by  Lemma~\ref{lem:no_monochromatic_element} we have
$$
M' = \begin{bmatrix}
I & B'\\
(B')^\intercal & I\\
\end{bmatrix}\,,
$$
where $B'$ is the submatrix\footnote{Notice that $B'$ is not always a true incidence matrix since there might be bichromatic edges $e$ such that only one endpoint of $e$ is bichromatic.} of the incidence matrix $B$ with rows indexed by bichromatic edges and columns indexed by bichromatic vertices. 

We use a Schur complement to express $\det M'$ in terms of a smaller matrix whose rows and columns are indexed by the bichromatic vertices. %In other words, we use row operations on the bottom right submatrix of $M'$ in order to zero the upper right submatrix of $M'$. 
We find
$$
\det M' = 
\det \begin{bmatrix}
I - B' (B')^\intercal& 0\\
(B')^\intercal & I\\
\end{bmatrix}
= \det (I - B' (B')^\intercal)\,.
$$

Let $G'$ denote the subgraph of $G$ whose edges are the bichromatic edges. %(By minimality of $M'$, the bichromatic vertices form a vertex cover of $G'$.) 
Let $G''$ denote the subgraph of $G'$ induced by the bichromatic vertices.
We denote by $\tilde{L}(G',G'')$ the matrix $B' (B')^\intercal - I$. Note that this matrix coincides with the matrix obtained from the adjacency matrix $A(G'')$ of $G''$ by changing the diagonal coefficient for each bichromatic vertex $v$ to $d_{G'}(v)-1$. In a formula,
$$
\tilde{L}(G',G'') = B' (B')^\intercal - I = A(G'') + \diag (d_{G'}(v_1)-1,\ldots,d_{G'}(v_k)-1)
$$
where $v_1, \ldots, v_k$ denote the bichromatic vertices.

%We point out the counterintuitive fact that $\det(B' (B')^\intercal - I)$ can be negative, for instance for certain cubic trees, see Example~\ref{ex:first}.

%Before stating our characterization of $\Delta(G)$ when $G$ is a forest, we establish some terminology that will be useful for dealing with determinants of matrices such as $B' (B')^\intercal - I$. Let $H$ denote any graph with vertices $w_1$, \ldots, $w_\ell$. Given a weighting $\nu \in \R^{V(H)}$ of the vertices of $H$, we let
% $$
% D(H,\nu) := \det(A(H) + \diag(\nu(w_1),\ldots,\nu(w_\ell)))\,.
% $$ 
% Observe that $D(H,\nu)$ is well defined since the determinant does not depend on the ordering of the vertices. 

From the discussion above, we directly obtain the following result. The final part of the theorem follows from the minimality of $M'$.

\begin{theorem} \label{thm:Delta_for_forests}
Let $G$ be any forest (with at least one vertex). Then
$$
\Delta(G) = \max \{|\det \tilde{L}(G',G'')| : G'' \stackrel{\mathrm{ind}}{\subseteq} G' \subseteq G\}
$$
where $G'$ is a subforest of $G$ and $G''$ is an induced subforest of $G'$. Moreover, in the formula above we can restrict to choices of subgraphs such that every edge of $G'$ has at least one endpoint in $G''$, and $d_{G'}(v) \geq 2$ for all $v \in V(G'')$ (provided that every component of $G$ has at least three vertices).
\end{theorem}

We conclude this section with a lemma that bounds $\Delta(G)$ in terms of the degree sequence of $G$. The lower bound is similar to Lemma~\ref{lemma:LB} in spirit, except that here we can be more explicit using the fact that $G$ is bipartite.

\begin{lemma} \label{lem:LB-UB}
Let $G$ be a forest, and let $d_1 \geq d_2 \geq \cdots \geq d_n$ denote the degrees of the vertices of $G$. For $d \in \mathbb{N}$, let $n_d$ denote the number of vertices of degree at least $d$ in $G$. Then we have
$$
\sqrt{\prod_{i=1}^{n_2} (d_i-1)} \leq \Delta(G) \leq \left(\frac{\sum_{i=1}^{n_2} d_i}{n_2}\right)^{n_2}\,.
$$
\end{lemma}

\begin{proof}
To prove the lower bound, consider any partition of the set of vertices of degree at least $2$ into two stable sets $S_1$ and $S_2$ (that is, any bicoloring of the corresponding induced subgraph of $G$). Notice that
$
\det \tilde{L}(G,G[S_1]) \cdot \det \tilde{L}(G,G[S_2]) = \prod_{i=1}^{n_2} (d_i-1)
$.
Hence, there exists some index $i \in [2]$ such that 
$
\det \tilde{L}(G,G[S_i]) \geq \sqrt{\prod_{i=1}^{n_2} (d_i-1)}
$.
We conclude by applying Theorem~\ref{thm:Delta_for_forests} with $G' := G$ and $G'' := G[S_i]$.

To prove the upper bound, let $G'$ and $G''$ be any choice of subgraphs of $G$ that achieve the maximum subdeterminant, see Theorem~\ref{thm:Delta_for_forests}. Thus, $\Delta(G) = |\det \tilde{L}(G',G'')|$. Again, let $v_1$, \ldots, $v_k$ denote the vertices of $G''$. Since $d_{G'}(v_i) \geq d_{G''}(v_i)$ for all $i \in [k]$, we have
%
% \begin{align*}
% \tilde{L}(G',G'') + I 
% &= A(G'') + \diag( d_{G'}(v_1),\ldots,d_{G'}(v_k) )\\
% &\succcurlyeq
% A(G'') + \diag( d_{G''}(v_1),\ldots,d_{G''}(v_k) )\\
% &\succcurlyeq 0\,.
% \end{align*}
$$
\tilde{L}(G',G'') + I 
\succcurlyeq
A(G'') + \diag( d_{G''}(v_1),\ldots,d_{G''}(v_k)) 
\succcurlyeq 0\,.
$$
This implies that the eigenvalues $\lambda_1 \geq \lambda_2 \geq \cdots \geq \lambda_k$ of $\tilde{L}(G',G'')$ are all larger or equal to $-1$. Let $\ell \in [k]$ be the last index for which $\lambda_\ell \geq 0$. We have
\begin{align*}\label{eq:max}
    \Delta(G)&=|\det\,\tilde{L}(G',G'')|=\left|\prod_{i=1}^k\lambda_i\right|\le\prod_{i=1}^\ell\lambda_i\le\left(\frac{\sum_{i=1}^\ell\lambda_i}{\ell}\right)^\ell\\
    &\le\left(\frac{(k-\ell)+\sum_{i=1}^k\lambda_i}{\ell}\right)^\ell %= \left(\frac{(k-\ell)+\sum_{i=1}^k(d_{G'}(v_i)-1)}{\ell}\right)^\ell
    =\left(\frac{\sum_{i=1}^kd_{G'}(v_i)}{\ell}-1\right)^\ell\\
    &\stackrel{(\star)}{\le}\left(\frac{\sum_{i=1}^kd_{G'}(v_i)}{k}\right)^k \le \max\left\{\left.\left(\frac{\sum_{i=1}^pd_i}{p}\right)^p\right|p\in[n],\ p \le n_2\right\}.
\end{align*}
Above, we use the AM-GM-inequality, the fact that the sum of the eigenvalues of $\tilde{L}(G',G'')$ equals its trace, $G''\subseteq G'\subseteq G$, and $d_{G'}(v_i)\ge2$ for all $i \in [k]$. It remains to show that $(\star)$ holds and that $p = n_2$ achieves the maximum in our final upper bound.

In order to prove $(\star)$, let $\bar{d} :=\frac{\sum_{i=1}^kd_{G'}(v_i)}{k}$ denote the average number of edges of $G'$ incident to a vertex of $G''$ and $k:=(1+\varepsilon)\ell$ for some $\varepsilon \ge 0$. Without loss of generality, we may assume $k > \ell$, i.e., $\varepsilon > 0$. By minimality of $M'$, we have $\bar{d} \ge 2$, see Theorem~\ref{thm:Delta_for_forests}.

Observe that $(\star)$ holds if and only if $(1+\varepsilon-\bar{d}^{-1})^{\frac1\varepsilon} \le \bar{d}$. In case $\bar{d} \ge \mathrm{e}$ we get $(1+\varepsilon-\bar{d}^{-1})^{\frac1\varepsilon} \le (1+\varepsilon)^{\frac1\varepsilon} \le \mathrm{e} \le \bar{d}$. In case
$2 \le \bar{d} < \mathrm{e}$, we get $\left(1+\varepsilon-\bar{d}^{-1}\right)^{\frac1\varepsilon}\le\left(1+\varepsilon-\mathrm{e}^{-1}\right)^{\frac{1}{\varepsilon}} \le 2 \le \bar{d}$, where the second to last inequality holds since $2^\varepsilon - \varepsilon \ge \frac{1 + \ln \ln 2}{\ln 2} \ge 1 - \mathrm{e}^{-1}$ for all $\varepsilon > 0$.

Finally let $1\le p < n_2$.
Define $x:=\sum_{i=1}^pd_i$ and observe that $x\ge2p$.
Then, 
\begin{align*}
    \frac{\left(\frac{x+d_{p+1}}{p+1}\right)^{p+1}}{\left(\frac{x}{p}\right)^p}&=\left(\frac{p}{p+1}\right)^p\cdot\left(\frac{(x+d_{p+1})^{p+1}}{(p+1)x^p}\right)\ge \mathrm{e}^{-1}\left(\frac{(x+2)^{p+1}}{(p+1)x^p}\right)\\
    &\ge \mathrm{e}^{-1}\left(\frac{x^{p+1}+2(p+1)x^p}{(p+1)x^p}\right)\ge \mathrm{e}^{-1}\left(\frac{4p+2}{p+1}\right)\ge 1.
\end{align*}
We conclude that $p = n_2$ maximizes the upper bound.
\qed \end{proof}
%
%\begin{align}
%\nonumber \Delta(G) &= | \det \tilde{L}(G',G'')| = \left| \prod_{i=1}^k \lambda_i \right| \leq \prod_{i=1}^k (2 + \lambda_i) \leq \left( \frac{\sum_{i=1}^k (2 + \lambda_i)}{k} \right)^k\\ 
%\nonumber &\leq \left( \frac{2k + \sum_{i=1}^k (d_{G'}(v_i) - 1)}{k} \right)^k =\left( 1 + \frac{\sum_{i=1}^k d_{G'}(v_i)}{k} \right)^k\\
%\label{eq:max} &\leq \max \left\{\left( 1 + \frac{\sum_{i=1}^p d_i}{p}\right)^{p} \mid p \in [n],\ d_p \geq 2\right\}\,.
%\end{align}
%
%Above, we used the inequality $|x| \leq 2+x$, which is valid for all $x \geq -1$, the AM-GM inequality, and also the facts that $G'' \subseteq G' \subseteq G$ and $d_{G'}(v_i) \geq 2$ for all $i \in [k]$.

%Next, notice that we have, for all indices $p < n_2$,
%$$
%\frac{1 + \frac{\sum_{i=1}^{p+1} d_i}{p+1}}{1 + \frac{\sum_{i=1}^p d_i}{p}}
%= \frac{p}{p+1} \cdot \frac{p+1 + \sum_{i=1}^{p+1} d_i}{p + \sum_{i=1}^p d_i}
%\ge \frac{p}{p+1}\,.
%$$
%Hence, we have
%
%\begin{align*}
%\frac{\left( 1 + \frac{\sum_{i=1}^{p+1} d_i}{p+1}\right)^{p+1}}{\left( 1 + %\frac{\sum_{i=1}^p d_i}{p}\right)^{p}}
%&= 
%\left( 1 + \frac{\sum_{i=1}^{p+1} d_i}{p+1}\right)
%\left(\frac{1 + \frac{\sum_{i=1}^{p+1} d_i}{p+1}}{1 + \frac{\sum_{i=1}^p d_i}{p}}\right)^{p}\\
%&\geq 
%\left( 1 + \frac{\sum_{i=1}^{p+1} d_i}{p+1}\right) \left( \frac{p}{p+1} \right)^p\\
%&\geq
%\mathrm{e}^{-1}
%\left( 1 + \frac{\sum_{i=1}^{p+1} d_i}{p+1}\right)
%\geq 3 \mathrm{e}^{-1} > 1\,.
%\end{align*}
%
%We conclude that the maximum in \eqref{eq:max} is attained for $p = n_2$.

We conclude this section by comparing the bounds we just proved on $\Delta(G)$.
Taking logarithms and assuming that $n_2 = O(n_3)$ (this is without loss of generality by Lemma~\ref{lem:long_induced_path}), we see that
$$
\log \left(\frac{\sum_{i=1}^{n_2} d_i}{n_2}\right)^{n_2} \leq 
O\left( \frac{\log d_\mathrm{max}}{\log d_\mathrm{min}}\right) \log \sqrt{\prod_{i=1}^{n_2} (d_i - 1)}\,,
$$
where $d_\mathrm{max}$ is the maximum degree and $d_\mathrm{min}$ is the minimum degree among all vertices of degree at least $3$. Hence, in general, the bounds are quasipolynomially related. If the degree distribution is not too spread out, they are even polynomially related. %For instance, if $G$ is a $r$-regular tree (that is, every inner vertex has the same degree $r$) then the bounds are polynomially related.

%Hence, in general, assuming $n_2 = O(n_3)$, our upper bound is quasipolynomially related to our lower bound. In certain cases, they are better related. For instance, if $G$ is a $r$-regular tree (that is, every inner vertex has the same degree $r$) then the bounds are polynomially related.

\section{Conclusions and future work} \label{sec:conclusion}

The total matching problem has a natural formulation as an IP. In this paper we investigate instances such that all subdeterminants of the constraint matrix are bounded (in absolute value) by a constant $\Delta \in \mathbb{Z}_{\ge 1}$.
Our main contributions include linear time algorithms for solving, as well as recognizing, such instances. %Our recognition algorithm either computes the maximum subdeterminant $\Delta(G)$, or reports that $\Delta(G) > \Delta$. 
Our algorithms rely on results on the structure of graphs $G$ with bounded $\Delta(G)$, which we establish here. Moreover, we prove further structural results about $\Delta(G)$ where $G$ is a forest.

It is not known in general how to check in polynomial time whether an integer matrix has all its subdeterminants bounded by a constant. In particular, no polynomial time algorithm appears to be known to test whether a given integer matrix is \emph{totally bimodular}, that is, has all its subdeterminants bounded by $2$ (see for instance~\cite{artmann_2017} for more information).

The main question that we leave for future work is that of determining the complexity of computing $\Delta(G)$. We believe that this is as difficult as computing the maximum subdeterminant of the incidence matrix of a graph. This last problem amounts to computing the odd cycle packing number $\ocp(G)$ graph $G$. This last invariant is NP-hard to compute, and even NP-hard to approximate~\cite{kawarabayashi_2010}. 

The constraint matrix $M(G)$ of the total matching IP contains the incidence matrix of $G$ as a submatrix, hence $\Delta(G)$ is at least $2^{\ocp(G)}$. However, for most graphs $\Delta(G)$ seems to be dominated by other structural aspects of the graph, in particular its degree sequence. 

Note that determining the complexity of computing $\Delta(G)$ remains open, even in the case where $G$ is a forest. Even though the formula given in Theorem~\ref{thm:Delta_for_forests} slightly simplifies the problem, it remains not trivial to understand precisely which edges and vertices should be selected in the corresponding bicoloring. For instance, there are examples showing that insisting to bicolor certain (high-degree) vertices can be detrimental for the subdeterminant, see Figure~\ref{fig:example}.

\tikzset{ 
	vertex/.style={
		draw,
		circle,
		inner sep=0,
        fill = black,
		minimum size=.15cm}
}
\begin{figure}[ht]
\centering
\begin{tikzpicture}
    \node[vertex, label=above right:$v$] (v) at (0,0) {};
    \node[vertex] (v1) at (0,1) {};
    \node[vertex] (v2) at (1,0) {};
    \node[vertex] (v3) at (-1,0) {};
    \node[vertex] (v11) at (-.5,1) {};
    \node[vertex] (v12) at (.5,1) {};
    \node[vertex] (v21) at (1,.5) {};
    \node[vertex] (v22) at (1.5,0) {};
    \node[vertex] (v31) at (-1,.5) {};
    \node[vertex] (v32) at (-1.5,0) {};

    \draw (v) -- (v1);
    \draw (v) -- (v2);
    \draw (v) -- (v3);
    \draw (v1) -- (v11);
    \draw (v1) -- (v12);
    \draw (v2) -- (v21);
    \draw (v2) -- (v22);
    \draw (v3) -- (v31);
    \draw (v3) -- (v32);
\end{tikzpicture}
\caption{A graph $G$ with $\Delta(G)=8$. In colorings where $v$ is bicolored, the  maximum subdeterminant is $4$. Hence it is optimal to leave $v$ uncolored.}
\label{fig:example}
\end{figure}
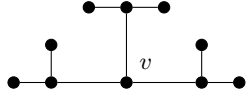

%For instance, it is possible to increase subdeterminants by leaving certain inner vertices uncolored, see Fig.~\ref{fig:example}.

When we consider general graphs, the picture becomes even less clear. Already for small graphs on up to six vertices, there is a large number of examples where certain elements have to be monochromatic in order to maximize the subdeterminant. %This leaves a more precise characterization of $\Delta(G)$ as an intriguing direction for further research. 
We note that there appears to be a strong connection between principal subdeterminants and spectral graph theory (as used in the proof of Lemma~\ref{lem:LB-UB}), in particular for forests and regular graphs.

\bibliographystyle{splncs04}
\bibliography{references}

\end{document}